\documentclass[12pt, a4paper]{amsart}

\usepackage[english]{babel}
\usepackage{amssymb,amsmath,amsthm,amsfonts}
\usepackage{graphicx}
\usepackage{xcolor}
\usepackage{enumerate}
\usepackage{tikz}

\newtheorem{theorem}{Theorem}
\newtheorem{example}[theorem]{Example}
\newtheorem{lemma}[theorem]{Lemma}
\newtheorem{remark}[theorem]{Remark}
\newtheorem{definition}[theorem]{Definition}

\numberwithin{equation}{section}
\numberwithin{theorem}{section}

\begin{document}

\title[Revolving structure of the L\'evy dragon]{On the revolving structure of the L\'evy dragon and its linear transforms}

\author[M. Gonzalez-Carriedo]{Miguel Gonzalez-Carriedo}
\address{Miguel Gonzalez-Carriedo, Department of Mathematics, University of North Texas, 1155 Union Circle, Denton, TX 76203-5017, USA}
\email{MiguelGonzalez-Carriedo@my.unt.edu}

\author[K. Kawamura]{Kiko Kawamura}
\address{Kiko Kawamura, Department of Mathematics, University of North Texas, 1155 Union Circle, Denton, TX 76203-5017, USA}
\email{Kiko.Kawamura@unt.edu}

\author[J. Leung]{Jonathan Leung}
\address{Jonathan Leung, Department of Mathematics, University of North Texas, 1155 Union Circle, Denton, TX 76203-5017, USA}
\email{JonathanLeung@my.unt.edu}

\author[R. D. Prokaj]{R. D\'{a}niel Prokaj}
\address{ R. D\'{a}niel Prokaj, Department of Mathematics, University of North Texas, 1155 Union Circle, Denton, TX 76203-5017, USA}
\email{Rudolf.Prokaj@unt.edu}

\begin{abstract}
  L\'evy's Dragon Curve is a well-known self-similar fractal, notable for its ability to tile the complex plane. We review a representation of the curve as a set of points given by complex power series satisfying a revolving condition, and study how this representation changes under linear transformations, while preserving its characteristic geometric properties.

  We introduce a labeled directed graph that encodes these series representations and show that this directed graph remains invariant under linear transformations, with only the labeling subject to change. Furthermore, we provide a geometric characterization of the resulting variations in graph labeling.
\end{abstract}



\maketitle
\thispagestyle{empty}

\section{Introduction}

L\'evy’s Dragon Curve (Figure \ref{lt89}) is a striking self-similar fractal, notable for its ability to tile the complex plane. It was introduced and studied in 1938 by P. L\'evy \cite{Levy-1938}. The curve is the unique attractor $\Lambda$ generated by the following planar iterated function system (IFS):
\begin{equation}\label{lt83}
\mathcal{F}=
  \bigg\{
    f_0(z)=\frac{1-i}{2}z,\quad 
    f_1(z)=\frac{1+i}{2}z+\frac{1-i}{2}  
  \bigg\}.
\end{equation}
That is, $\Lambda$ is the unique non-empty compact set satisfying 
\begin{equation}\label{lt84}
  \Lambda = f_0(\Lambda)\cup f_1(\Lambda).
\end{equation} 

Figure \ref{fig:construction} shows how $\Lambda$ is obtained recursively from the isosceles triangle $A_0$ with vertices $0,1$, and $(1-i)/2$ under the action of the IFS \eqref{lt83}. From this construction, it is immediate that the area enclosed by $\Lambda$ coincides with that of $A_0$. Indeed, L\'evy proved that $\Lambda$ is a space-filling curve. 

\begin{figure}[h]
\begin{center}
\begin{picture}(320,180)(20,-120)
	\put(14,50){\line(1,-1){40}}
	\put(14,50){\line(1,0){80}}
	\put(54,10){\line(1,1){40}}
	\put(34,40){\vector(-1,-1){10}}
	\put(74,40){\vector(1,-1){10}}
	\put(12,30){\makebox(0,0)[tl]{$f_0$}}
	\put(86,30){\makebox(0,0)[tl]{$f_1$}}	
	\put(49,-10){\makebox(0,0)[tl]{$A_0$}}
	\put(116,50){\line(1,-1){40}}
	\put(156,10){\line(1,1){40}}
	\put(116,10){\line(0,1){40}}
	\put(116,10){\line(1,0){80}}
	\put(196,50){\line(0,-1){40}}
	\put(121,20){\vector(-1,0){10}}
	\put(83,15){\makebox(0,0)[tl]{$f_0 \circ f_0$}}
	\put(136,15){\vector(-1,-2){5}}
	\put(111,0){\makebox(0,0)[tl]{$f_1 \circ f_0$}}
	\put(176,15){\vector(1,-2){5}}
	\put(168,0){\makebox(0,0)[tl]{$f_0 \circ f_1$}}
	\put(191,20){\vector(1,0){10}}
	\put(198,15){\makebox(0,0)[tl]{$f_1 \circ f_1$}}
	\put(156,-10){\makebox(0,0)[tl]{$A_1$}}
	\put(247,10){\line(0,1){40}}
	\put(247,10){\line(1,0){80}}
	\put(327,50){\line(0,-1){40}}
	\put(287,-10){\makebox(0,0)[tl]{$A_2$}}
	\put(247,10){\line(-1,1){20}}
	\put(227,30){\line(1,1){20}}
	\put(247,10){\line(1,-1){20}}
	\put(267,-10){\line(1,1){20}}
	\put(287,10){\line(1,-1){20}}
	\put(307,-10){\line(1,1){20}}
	\put(327,50){\line(1,-1){20}}
	\put(347,30){\line(-1,-1){20}}
	\put(90,-115){\makebox(0,0)[tl]{$A_3$}}
	\put(35,-60){\line(1,1){20}}
	\put(55,-80){\line(-1,1){20}}
	\put(55,-80){\line(1,-1){20}}
	\put(75,-100){\line(1,1){20}}
	\put(95,-80){\line(1,-1){20}}
	\put(115,-100){\line(1,1){20}}
	\put(135,-40){\line(1,-1){20}}
	\put(155,-60){\line(-1,-1){20}}
	\put(55,-80){\line(-1,0){20}}
	\put(35,-80){\line(0,1){40}}
	\put(35,-40){\line(1,0){20}}
	\put(55,-100){\line(0,1){20}}
	\put(55,-100){\line(1,0){80}}
	\put(95,-80){\line(0,-1){20}}
	\put(135,-80){\line(0,-1){20}}
	\put(135,-80){\line(1,0){20}}
	\put(155,-40){\line(0,-1){40}}
	\put(155,-40){\line(-1,0){20}}
	\put(260,-115){\makebox(0,0)[tl]{$A_4$}}
	\put(225,-40){\line(-1,0){20}}
	\put(205,-80){\line(0,1){40}}
	\put(205,-80){\line(1,0){20}}
	\put(225,-100){\line(0,1){20}}
	\put(225,-100){\line(1,0){80}}
	\put(305,-80){\line(0,-1){20}}
	\put(305,-80){\line(1,0){20}}
	\put(325,-40){\line(0,-1){40}}
	\put(325,-40){\line(-1,0){20}}
	\put(225,-80){\line(-1,-1){10}}
	\put(215,-90){\line(-1,1){10}}
	\put(205,-80){\line(-1,1){10}}
	\put(195,-50){\line(1,1){10}}
	\put(205,-60){\line(-1,1){10}}
	\put(195,-70){\line(1,1){10}}
	\put(205,-40){\line(1,1){10}}
	\put(215,-90){\line(1,-1){10}}
	\put(215,-30){\line(1,-1){10}}
	\put(225,-100){\line(1,-1){10}}
	\put(235,-110){\line(1,1){10}}
	\put(245,-100){\line(1,-1){10}}
	\put(255,-110){\line(1,1){10}}
	\put(265,-100){\line(1,-1){10}}
	\put(265,-100){\line(0,1){20}}
	\put(265,-100){\line(1,1){10}}
	\put(275,-90){\line(-1,1){10}}
	\put(265,-80){\line(-1,-1){10}}
	\put(255,-90){\line(1,-1){10}}
	\put(275,-110){\line(1,1){10}}
	\put(285,-100){\line(1,-1){10}}
	\put(295,-110){\line(1,1){30}}
	\put(305,-40){\line(1,1){10}}
	\put(325,-60){\line(1,-1){10}}
	\put(335,-70){\line(-1,-1){10}}
	\put(325,-40){\line(1,-1){10}}
	\put(335,-50){\line(-1,-1){10}}
	\put(315,-30){\line(1,-1){10}}
	\put(315,-90){\line(-1,1){10}}
	\end{picture}
\caption{The first five steps of the construction of L\'evy's dragon curve.}
\label{fig:construction}
\end{center}
\end{figure}
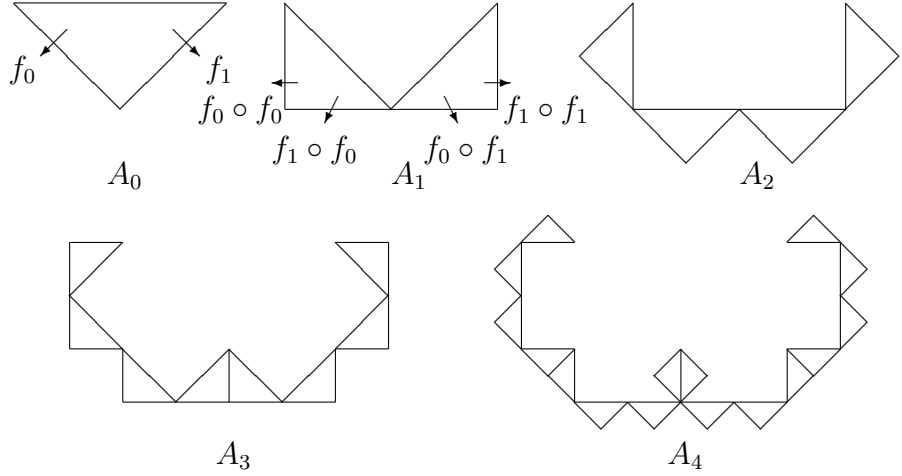

A natural question is whether there exist alternative ways to represent 
L\'evy's dragon, and in particular, whether it admits an explicit description. From its construction, it is clear that every point $z\in \Lambda$ can be encoded by at least one infinite sequence $(k_i)_{i=1}^{\infty} \in \{0,1\}^{\mathbb{N}}$ such that 
\begin{equation*}
    z= \lim_{n \to \infty}f_{k_1} \circ f_{k_2} \circ \dots \circ f_{k_n} (0).
\end{equation*}

This naturally leads to the question of whether the L\'evy dragon admits another, more explicit representation. Since each point $z\in\Lambda$ is a complex number, it is reasonable to ask whether $\Lambda$ can be described in terms of a complex power series.

\medskip

Eleven years before Hutchinson’s seminal 1981 paper, which introduced the modern framework of iterated function systems, C.~Davis and D.~E.~Knuth developed the notion of \texttt{revolving representations} for Gaussian integers. Specifically, for any Gaussian integer $z=x+iy$ where $x, y \in \mathbb{Z}$,  there exists a finite sequence $(\delta_{0}, \delta_{1},\dots \delta_{N})$ such that  
\begin{equation*}
  z=\sum_{n=0}^{N} \delta_{N-n}(1+i)^{n}:=(\delta_0, \delta_1, \cdots, \delta_N)_{1+i}, 
\end{equation*}
where each digit $\delta_n \in \{0, 1, -1, i, -i\}$ satisfies the condition that the nonzero digits follow the unit circle in a counterclockwise fashion, as illustrated in Figure \ref{lt81}. They referred to this rule as the \texttt{revolving condition}.

\begin{figure}[h]
  \begin{center}
  \begin{tikzpicture}[scale=2]
    \draw[->] (-1.1,0) -- (1.1,0);
    \draw[->] (0,-1.1) -- (0,1.1);
    \draw [opacity=0.3](0,0) circle (1);
    \foreach \angle in {37.5, 127.5, 217.5, 307.5} {
        \draw[<-,line width=1pt] (\angle:1) arc (\angle:\angle+15:1);
    }
    \node at (1.3,0) {$1$};
    \node at (-1.3,0) {$-1$};
    \node at (0,1.3) {$i$};
    \node at (0,-1.3) {$-i$};
  \end{tikzpicture}
  \end{center} 
  \caption{The unit circle in the complex plane with rotation $\theta=-\frac{\pi}{2}$.}
  \label{lt81}
\end{figure}
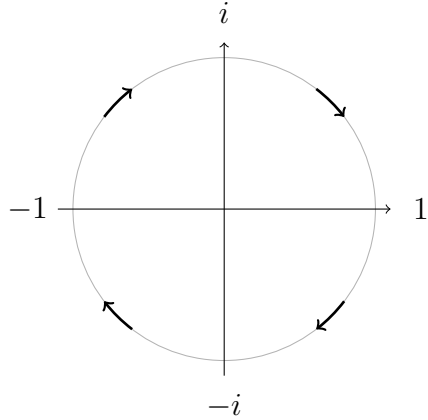 

They proved that every Gaussian integer possesses exactly four revolving representations, distinguished by the value of the rightmost nonzero digit, which may be $1, -1, i$ or $-i$. For example, 
\begin{align*}
3+i &=((-1), i, 0, 1, 0)_{1+i} =(i, 1, (-i), 0, (-1), i, 0)_{1+i} \\
    &=(i, 1, (-i), 0, 0, (-1), 0)_{1+i} =((-1), i, 1, (-i), 0)_{1+i}
\end{align*}

Let $W_{\theta}$ denote the set of all revolving sequences with a fixed angle $\theta$. In 1987, M.~Mizutani and S.~Ito~\cite{Ito-1987} considered the following set of points. 
\begin{equation}
\label{Mizutani-Ito-X}
  P:=\left\{ \sum_{n=1}^{\infty} \delta_{n}(1+i)^{-n}: \delta_{j_1}=1, (\delta_{1}, \delta_{2}, \dots ) \in W_{-\pi/2} \right\},
\end{equation}
where $\delta_j \in \{0, 1, -1, i, -i\}$ and $j_1= \min\{j: \delta_j \neq 0\}$. 
They proved that $P$ is a paper-folding dragon, which is a tiling fractal (Figure \ref{lt89}).

In the same paper, they also proposed an interesting conjecture. Suppose that the digits $\delta_{n}$ move counterclockwise on the unit circle rather than clockwise. Their computer simulations suggested that the corresponding set is a L\'evy dragon. More specifically,
\begin{equation}
\label{lt82}
  \Lambda=\left\{ \sum_{n=1}^{\infty} \delta_{n}(1+i)^{-n}: \delta_{j_1}=1, (\delta_{1}, \delta_{2}, \dots ) \in W_{\pi/2} \right\},
\end{equation}
where $\delta_j \in \{0, 1, -1, i, -i\}$ and $j_1= \min\{j: \delta_j \neq 0\}$. 

Fifteen years later, Kawamura \cite{Kawamura-2002} proved that their conjecture is correct using a functional equation approach. More recently, Kawamura and Allen \cite{Kawamura-Allen-2021} introduced the notion of {\it generalized revolving sequences}, in which the rotation angle $\theta$ is allowed to vary more generally. They showed that certain self-similar attractors generated by iterated function systems consisting of two similarities with a common scale factor---one of which includes a rotation through angle $\theta$ ---admit natural complex power series representations satisfying a generalized revolving condition. Their result may be viewed as a natural generalization of the earlier work of Mizutani and Ito.

It is important to note that such representations are intrinsically tied to the underlying IFS. For example, the representation in \eqref{lt82} corresponds specifically to the system given in \eqref{lt83}. When $\Lambda$ is translated in the complex plane, the resulting set still retains the characteristic geometric properties of L\'evy’s Dragon Curve, but it is generated by a different IFS. This naturally leads to the following question:
\medskip

\textit{Can every translation of L\'evy’s Dragon Curve be represented by a complex power series satisfying a revolving condition?}

\medskip
To address this question, we first introduce a general notion of revolving numbers by introducing a labeled directed graph $\mathcal{G}=\{\mathcal{V},\mathcal{E},\mathcal{L}\}$ that takes the place of the revolving condition.
Directed graphs are determined by a set of nodes $\mathcal{V}=\{v_1,\dots, v_M\}$ and a set of directed edges $\mathcal{E}$. Each edge $e\in\mathcal{E}$ is an ordered pair of two nodes. When the graph is also labeled, we have a mapping $\mathcal{L}$ that acts on $\mathcal{V}$ and labels the nodes with a finite alphabet.
In this paper, we will argue that our definition captures the structural intricacies better than the revolving conditions.
\begin{definition}
  We say that $x\in\mathbb{C}$ is a \texttt{revolving number} with base $b\in\mathbb{C}$, digit set $\mathcal{D}=\{\delta_1,\dots,\delta_d\}$ and labeled directed graph $\mathcal{G}=(\mathcal{V},\mathcal{E}, \mathcal{L})$ if the following hold
  \begin{enumerate}
    \item $\mathcal{L}:\mathcal{V}\to\mathcal{D}$ labels the nodes with the digits;
    \item $x=c\cdot\sum_{n=1}^{\infty} b^{-n}\delta_{j_n}$ for a digit sequence $(\delta_{j_n})_{n\geq 1}\in\mathcal{D}^{\mathbb{N}}$ and some constant $c\in\mathbb{C}$;
    \item There exists a $(v_{j_n})_{n\geq 1}\in\mathcal{V}^{\mathbb{N}}$ sequence of nodes such that \\ $\mathcal{L}(v_{j_n})=\delta_{j_n}$ and $(v_{j_n},v_{j_{n+1}})\in\mathcal{E}$ for every $n\geq 1$.
  \end{enumerate}
 
\end{definition}

Let $F$ be the attractor of a planar IFS. If all points of $F$ are revolving numbers of base $b\in\mathbb{C}$, digit set $\mathcal{D}=\{\delta_1,\dots,\delta_d\}$ and labeled directed graph $\mathcal{G}$, then we say that $F$ has a \texttt{revolving structure} with parameters $(b, \mathcal{D}, \mathcal{G})$. 

\begin{definition}
  Let $F_1$ and $F_2$ be two sets on the plane of revolving structures with parameters $(b_1,\mathcal{D}_1,\mathcal{G}_1)$ and $(b_2,\mathcal{D}_2,\mathcal{G}_2)$. If the labeled digraphs $\mathcal{G}_1$ and $\mathcal{G}_2$ are defined by the same node and edge sets, hence the only difference between them is the labeling, then we say that $F_1$ and $F_2$ have the \texttt{same revolving pattern}. 
\end{definition}

\begin{theorem}[Main Theorem]\label{lt93}
  Let $S$ be a non-singular linear transformation on $\mathbb{C}$ and $S(\Lambda)$ be the transformed L\'evy dragon. Then, $S(\Lambda)$ also admits a revolving structure.
  Further, $\Lambda$ and $S(\Lambda)$ have the same revolving pattern.

\end{theorem}

We note that while the base and the directed graph that determines the revolving structure doesn't change under linear transformations, the digit set and thus the labeling of the nodes might be different.

\section{Revolving structure of the L\'evy dragon}

\begin{figure}
  \centering 
  \includegraphics[width=.48\linewidth]{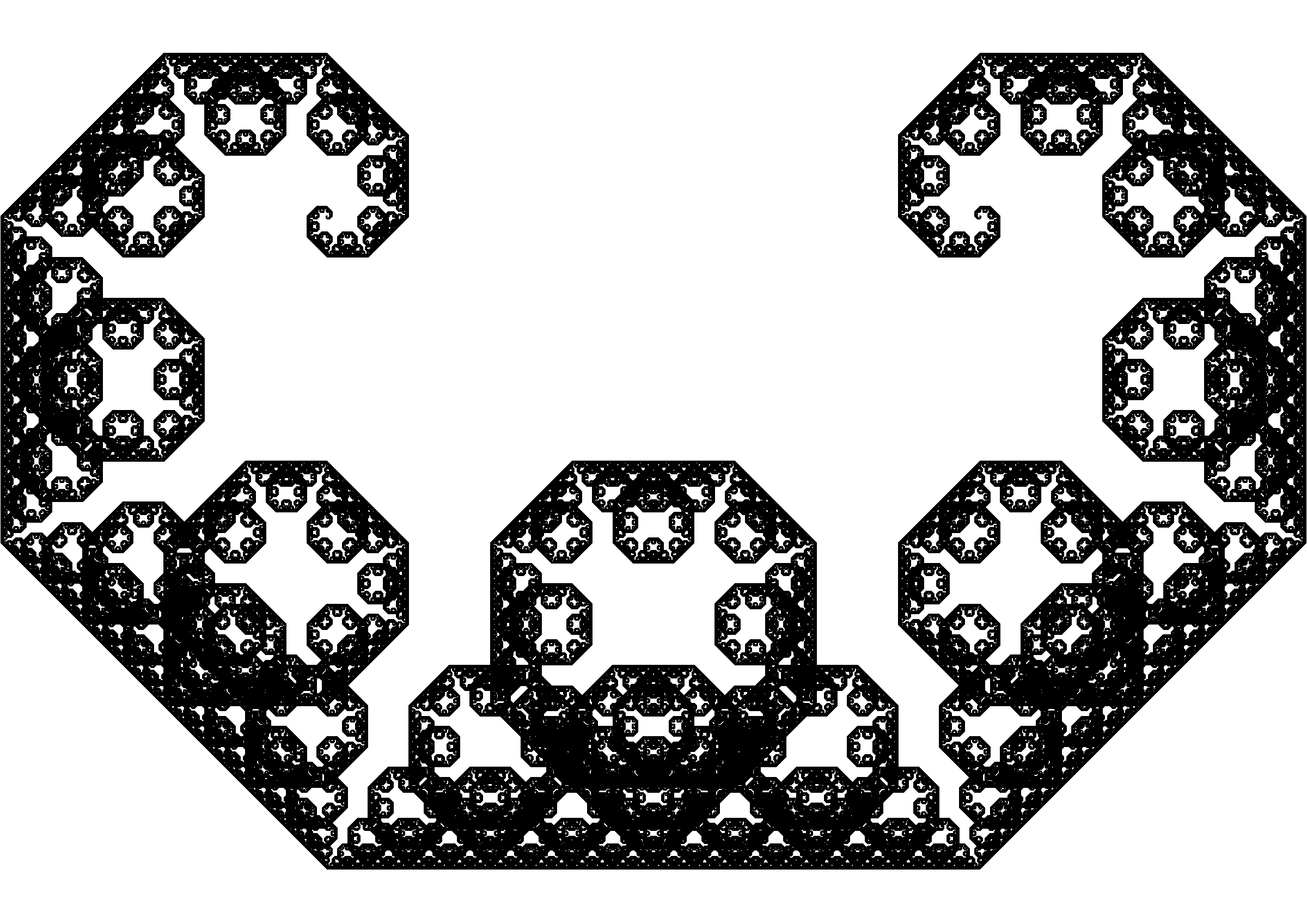}\hfill
  \includegraphics[width=.48\linewidth]{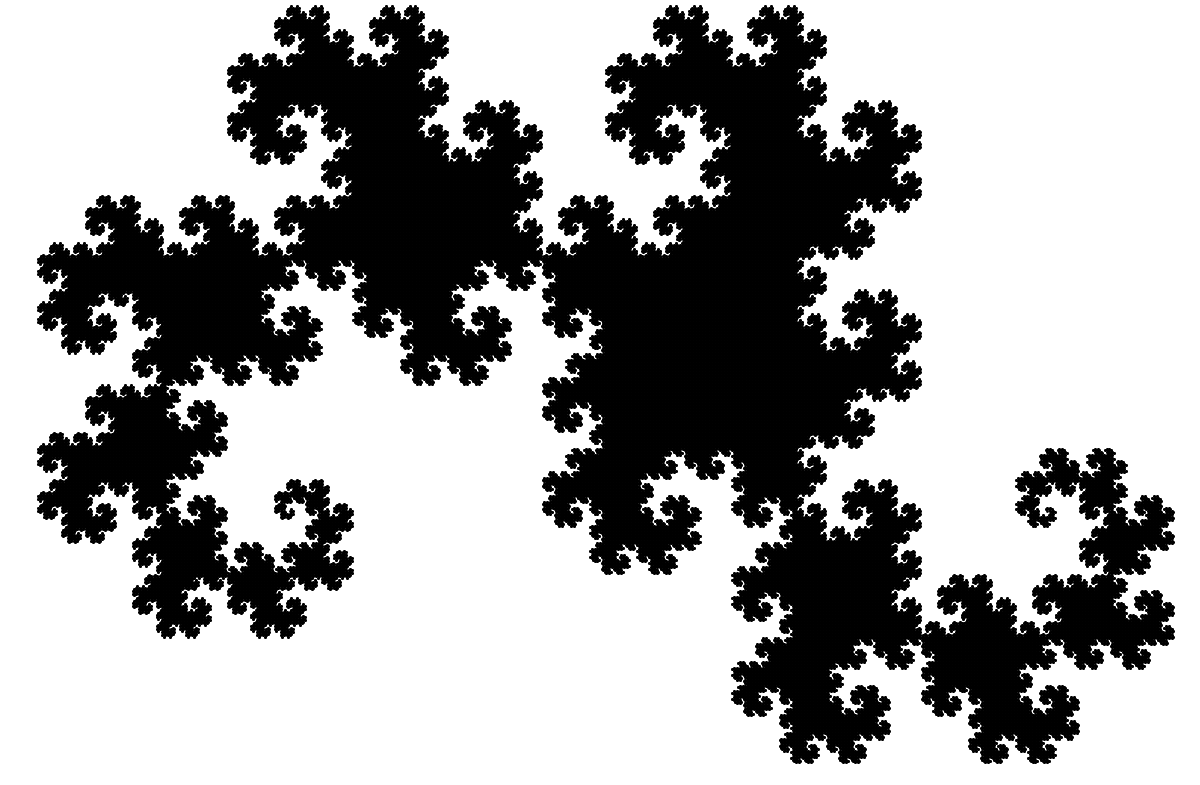}
  \caption{Two tiling fractals: the L\'evy dragon curve on the left and the paper-folding dragon on the right.}\label{lt89}
\end{figure}

Let $\Sigma=\{0,1\}^{\mathbb{N}}$ be the set of all $0-1$ sequences. We call $\Sigma$ equipped with the left-shift operator $\sigma:\Sigma\to\Sigma,\quad \sigma(\omega_1\omega_2\omega_3\dots)=\omega_2\omega_3\dots$ the \texttt{symbolic space}. We define the \texttt{canonical projection} from $\Sigma$ to the L\'evy dragon $\Lambda$ as
\[
  \Pi:\Sigma\to\Lambda,\quad 
  \Pi(\pmb{\omega})=\lim_{n\to\infty} f_{\pmb{\omega}|_n}(0),
\]
where $\pmb{\omega}|_n=\omega_1\omega_2\dots\omega_n$ and $f_{\pmb{\omega}|_n}=f_{\omega_1}\circ\dots\circ f_{\omega_n}$ for every $\pmb{\omega}\in\Sigma$.

Since $\Lambda$ satisfies \eqref{lt84}, the canonical projection $\Pi$ must satisfy the following equalities for every $\pmb{\omega}\in\Sigma$.
\begin{equation}\label{lt99}
    \Pi(\pmb{\omega})=\begin{cases}
        \frac{1-i}{2}\Pi(\sigma\pmb{\omega}) &\mbox{, if } \omega_1=0 \\
        \frac{1+i}{2}\Pi(\sigma\pmb{\omega})+\frac{1-i}{2} &\mbox{, if } \omega_1=1
    \end{cases}
\end{equation}
We may rewrite \eqref{lt99} in one line 
\begin{equation}\label{lt95}
    \Pi(\pmb{\omega})=\left(\frac{1-i}{2}\right)^{1-\omega_1}\left(\frac{1+i}{2}\right)^{\omega_1}\Pi(\sigma\pmb{\omega})+\frac{1-i}{2}\omega_1.
\end{equation}
Let $q_k(\pmb{\omega})=\sum_{i=1}^{k} \omega_i$ be the sum of the first $k$ digits of $\pmb{\omega}$. Pick an arbitrary $k\in\mathbb{N}$. We iterate \eqref{lt99} $k$ times to write $\Pi$ in the following form 
\begin{equation}\label{lt98}
    \Pi(\pmb{\omega})=\left(\frac{1-i}{2}\right)^k i^{q_k(\pmb{\omega})} \Pi(\sigma(\pmb{\omega})) +
    \sum_{n=1}^{k}\omega_n\left(\frac{1-i}{2}\right)^n i^{q_{n-1}(\pmb{\omega})}.
\end{equation}
Now, we take the limit of both sides in \eqref{lt98} as $k\to\infty$. Since $i^{q_k(\pmb{\omega})} \Pi(\sigma(\pmb{\omega}))$ is bounded and $\left|\frac{1-i}{2}\right|<1$, the first term converges to $0$. Therefore,
\begin{equation}\label{lt97}
    \Pi(\pmb{\omega})=\sum_{n=1}^{\infty} \left(\frac{1-i}{2}\right)^n i^{q_{n-1}(\pmb{\omega})} \omega_n = \sum_{n=1}^{\infty} (1+i)^{-n} i^{q_{n-1}(\pmb{\omega})} \omega_n .
\end{equation}
It shows that the points of the L\'evy dragon are the revolving numbers of base $1+i$, digit set $\{0,1,i,-1,-i\}$ and labeled directed graph $\mathcal{G}_{\rm L}$ depicted on Figure \ref{lt96}. Observe that the digit $0$ appears on several nodes; it might seem redundant, but we must keep track of the last non-zero digit in the revolving sequence, as the numbers $1,i,-1,-i$ can only follow one another in a certain order.

With the help of our symbolic space $(\Sigma,\sigma)$, we can code every point of the attractor with $0-1$ sequences. The revolving structure gives us a different coding of the points of $\Lambda$, using complex digits and a special pattern described by the labeled digraph $\mathcal{G}_{\rm L}$. For example, consider the point $x\in\Lambda$ coded by $\pmb{\omega}=(1,0,0,1,1,0,1,1,0,0,..)\in\Sigma$. Then, the corresponding digit sequence $(\delta_{j_n})_{n\geq 1}$ that gives $x$ as a revolving number is $(1,0,0,i,-1,0,-i,1,0,0,..)$.

In fact, both codings can be thought of as walks on $\mathcal{G}_{\rm L}$ if we suitably label the edges by $0$-s and $1$-s: along any directed path, the sequence of edges give us the coding on the symbolic space, while the sequence of nodes give us the revolving representation. To highlight this connection, we colored the edges of $\mathcal{G}_{\rm L}$ accordingly. We note that two directed edges start at each node; exactly one of each color.

It is natural to ask why we need a directed-graph of $8$ nodes to describe a revolving structure of only $5$ distinct digits. To answer this question, we will investigate the change of revolving structures under linear transformations.

\begin{figure}
  \includegraphics[scale=0.8]{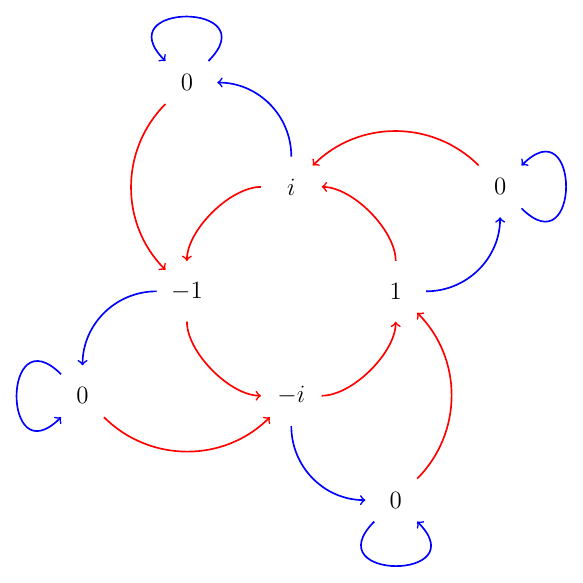}
  \caption{The labeled directed graph $\mathcal{G}_{\rm L}$ that determines the points of the L\'evy dragon as revolving numbers. By considering a sequence of directed edges instead of the corresponding sequence of vertices, one may re-obtain the coding on the symbolic space: Blue edges correspond to $0$, red edges correspond to $1$ in $\Sigma$.}\label{lt96}
\end{figure}

\section{Linear transformations}

Before proving our results for an arbitrary linear transformation, we show a simple example where the linear transformation $S$ is a translation on the complex plane. A more detailed analysis of this example can be found in \cite{Leung-2026}.
\begin{example}[Translated IFS]\label{lt88}
  Let $S(z)=z-\frac{1-i}{2}$ be a translation on the complex plane. The translated L\'evy dragon $S(\Lambda)$ is also a self-similar set, so there must exist contracting similitudes $\widehat{f}_0$ and $\widehat{f}_1$ for which 
  \begin{equation}\label{lt87}
    S(\Lambda)=\widehat{f}_0(S(\Lambda))\cup\widehat{f}_1(S(\Lambda)).
  \end{equation}
  To find these functions, observe that they perform the same action as $f_0$ and $f_1$ did on $\Lambda$. Thus, they must satisfy
  \[
    \widehat{f}_0(z)=S\circ f_0\circ S^{-1}(z),\quad
    \widehat{f}_1(z)=S\circ f_1\circ S^{-1}(z).
  \]
  That is, by substituting $f_0, f_1$ and $S$ into these formulas, 
  \begin{align*}
    \widehat{f}_0(z)&=\frac{1-i}{2}\left(z+\frac{1-i}{2}\right)-\frac{1-i}{2}= 
    \frac{1-i}{2}z-\frac{1}{2}\\
    \widehat{f}_1(z)&=\frac{1+i}{2}\left(z+\frac{1-i}{2}\right)+\frac{1-i}{2}-\frac{1-i}{2}=
    \frac{1+i}{2}z+\frac{1}{2}
  \end{align*}

  Let us write $\widehat{\Pi}:\Sigma\to\mathbb{C}$ for the canonical projection of this new, translated IFS. It follows from \eqref{lt87} that for all $\pmb{\omega}=\omega_1\omega_2\dots\in\Sigma$
  \[
    \widehat{\Pi}(\pmb{\omega})=\begin{cases}
        \frac{1-i}{2}\widehat{\Pi}(\sigma\pmb{\omega})-\frac{1}{2} &\mbox{, if } \omega_1=0 \\
        \frac{1+i}{2}\widehat{\Pi}(\sigma\pmb{\omega})+\frac{1}{2} &\mbox{, if } \omega_1=1.
    \end{cases}
  \]
  Analogously to \eqref{lt97}, we obtain
  \[
    \widehat{\Pi}(\pmb{\omega})=\frac{1+i}{2}\sum_{n=1}^{\infty} (1+i)^{-n}i^{q_{n-1}(\pmb{\omega})}(-1)^{1-\omega_n}.
  \] 
  It shows that the points of the translated L\'evy dragon are the revolving numbers of base $1+i$ and digit set $\{1,i,-1,-i\}$. It is straightforward that digits must follow each other according to the labeled directed graph on Figure \ref{lt86}. Since $q_0(\pmb{\omega})=0$ for every word $\pmb{\omega}$, the first digit $\delta_1$ is either $1$ or $-1$. More precisely, the initial node $1$ is located on the right of the circle, and $(-1)$ is the node with a self-loop at the bottom of the Figure \ref{lt86}.
	
  
  Observe that this graph looks identical to $\mathcal{G}_{\rm L}$, only the labels of the nodes are different. Therefore, $\Lambda$ and $S(\Lambda)$ has the same revolving pattern.

\end{example}

\begin{figure}
  \centering 
  \includegraphics[scale=0.8]{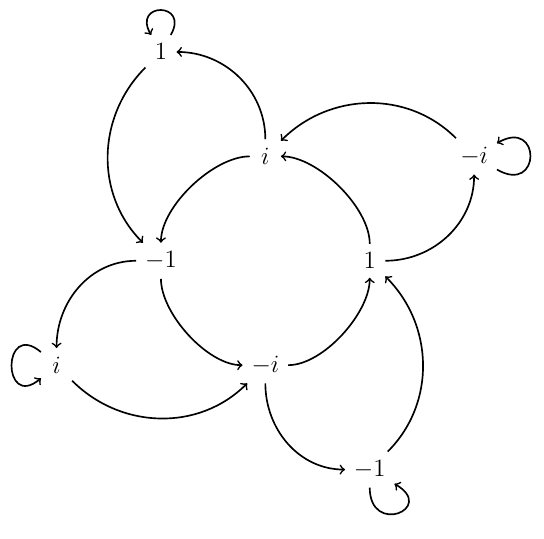}
  \caption{The labeled digraph that determines the revolving structure of the translated L\'evy dragon in Example \ref{lt88}.}\label{lt86}
\end{figure}

We note that it is rather difficult to find the similarity between the revolving structures of the sets $\Lambda$ and $S(\Lambda)$ without the labeled directed graphs we introduced. Relying on the terminology of Davis-Knuth, one must compare the revolving conditions. However, since $\Lambda$ and $S(\Lambda)$ has a different number of digits in their revolving structures, the revolving conditions seem to describe completely different dynamics.

The revolving condition of $\Lambda$ is analogous to the revolving condition of Gaussian integers: the non-zero digits must follow the cyclic pattern
\[
  \dots \to 1 \to i \to -1 \to -i \to 1 \to \dots,
\]
with arbitrary many $0$-s between them. The revolving condition of $S(\Lambda)$ is much more complicated. For instance, $1$ can be followed by both $i$ and $-i$, depending on the past digits. Sometimes we move clockwise, other times we move counterclockwise around the origin, which makes the revolving condition rather technical to formalize.

Now we give the proof of our main theorem.
\begin{proof}[Proof of Theorem~\ref{lt93}]
    Let $S(z)=\lambda z+\tau$ be a non-singular linear transformation on $\mathbb{C}$. Then $S(\Lambda)$ is the attractor of the IFS 
    \begin{equation}
      \widehat{\mathcal{F}}=\left\{
        \widehat{f}_0(z)= S\circ f_0\circ S^{-1}(z), \quad
        \widehat{f}_1(z)= S\circ f_1\circ S^{-1}(z)
      \right\}.
    \end{equation}
    Using that $S^{-1}(z)=\frac{z-\tau}{\lambda}$, a simple calculation gives
    \begin{equation}
      \widehat{\mathcal{F}}=\left\{
        \widehat{f}_0(z)= \frac{1-i}{2}z+\frac{1+i}{2}\tau, \quad
        \widehat{f}_1(z)= \frac{1+i}{2}z+\frac{1-i}{2}(\lambda+\tau) 
      \right\}.
    \end{equation}

    Analogously to \eqref{lt98}, the canonical projection $\widehat{\Pi}$ of $\widehat{\mathcal{F}}$ can be written in the form
    \begin{align*}
      \widehat{\Pi}(\pmb{\omega}) &=\left(\frac{1-i}{2}\right)^k i^{q_k(\pmb{\omega})}\widehat{\Pi}(\sigma^k\pmb{\omega})\\ 
      &+ \sum_{n=1}^{k} \left(\frac{1-i}{2}\right)^{n-1} i^{q_{n-1}(\pmb{\omega})} \left(\frac{1+i}{2}\tau +\omega_n\left(\frac{1-i}{2}\lambda -i\tau\right)\right).
    \end{align*}
    Then, we take the limit on both sides as $k\to\infty$ to obtain
    \[
      \widehat{\Pi}(\pmb{\omega})=\sum_{n=1}^{\infty} \left(1+i\right)^{-n} i^{q_{n-1}(\pmb{\omega})} \left( i\tau + \omega_n\left(\lambda +(1-i)\tau\right)\right).
    \]
    It follows that the points of the transformed L\'evy dragon $S(\Lambda)$ are the revolving numbers of base $1+i$ and digit set 
    \[
      \left\{ 
        \lambda+\tau,\: i(\lambda+\tau),\: -(\lambda+\tau),\: -i(\lambda+\tau),\: 
        \tau,\: i\tau,\: -\tau,\: -i\tau
      \right\}.
    \]
    Further, the digits must follow each other according to the directed graph on Figure \ref{lt94}. Note that the first digit $\delta_1$ is either $\lambda+\tau$ or $i\tau$. 

    \begin{figure}
    \includegraphics[scale=0.8]{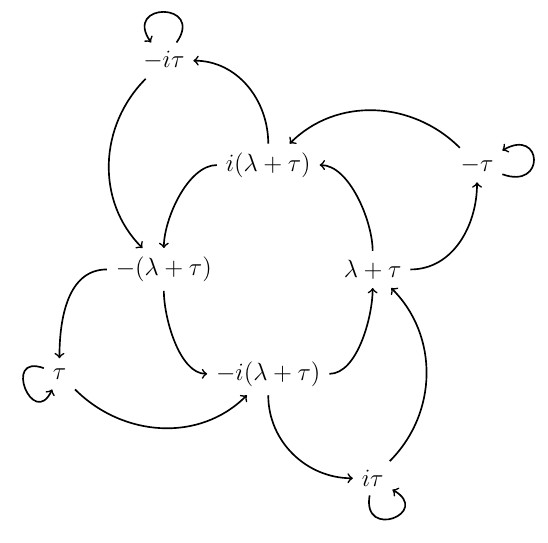}
    \caption{The directed graph of the revolving structure of $S(\Lambda)$.}\label{lt94}
    \end{figure}

\end{proof}

It follows from the proof of Theorem \ref{lt93} that the digit set of $\Lambda$ changes under the linear transformation $S(z)=\lambda z+\tau$, as the new digits depend on the parameters $\lambda$ and $\tau$. It also follows that $S(\Lambda)$ has at most 8 different digits in its revolving structure. However, since the digits depend on only 2 parameters, the number of digits cannot be any arbitrary integer between 1 and 8.
\begin{lemma}\label{lt90}
  Let $S(z)=\lambda z+\tau$ be a non-singular linear mapping on the complex plane and $S(\Lambda)$ be the corresponding transformed L\'evy dragon. Write $d$ for the number of digits in the revolving structure of $S(\Lambda)$. Then, 
  \begin{enumerate}[(a)]
    \item $d=4$ if and only if $\exists k\in\mathbb{N}:\: i^k(\lambda+\tau)=\tau$;
    \item $d=5$ if and only if $\tau=0$ or $\tau=-\lambda$;
    \item otherwise, $d=8$. 
  \end{enumerate}

\end{lemma}

\begin{proof}
  We showed in the proof of Theorem \ref{lt93} that the digits of $S(\Lambda)$ are 
  \begin{equation}\label{lt92}
      \lambda+\tau,\: i(\lambda+\tau),\: -(\lambda+\tau),\: -i(\lambda+\tau),\: 
      \tau,\: i\tau,\: -\tau,\: -i\tau.
  \end{equation}

  Since $S$ in non-singular, $\lambda\neq 0$. Observe that every value in \eqref{lt92} is of the form $(\lambda+\tau)\cdot i^k$ or $\tau\cdot i^k$ for some $k\in\{0,1,2,3\}$. Clearly, if either $\tau=0$ or $\tau=-\lambda$, four digits will be 0, hence we only have 5 distinct digits.
  For instance, if $\tau=0$, the digits are 
  \[
        \lambda+\tau, i(\lambda+\tau), -(\lambda+\tau), -i(\lambda+\tau), 0.
  \]

  Assume now that $\tau\not\in\{0,-\lambda\}$. We define the sets 
  \begin{align*}
    N_0 &=\bigg\{\tau\cdot i^k:\: k\in\{0,1,2,3\}\bigg\}, \\
    N_1 &=\bigg\{(\tau+\lambda)\cdot i^k:\: k\in\{0,1,2,3\}\bigg\}.
  \end{align*}
  It follows that there are $4$ different numbers both in $N_0$ and $N_1$, hence we have at least $4$ digits. We just need to find under which condition $N_0\cap N_1\neq\emptyset$, otherwise there are $8$ distinct digits.

  If $N_0\cap N_1\neq\emptyset$, then there must exist $x\in N_0\cap N_1$. By the definition of $N_0$ and $N_1$, 
  \[
    \exists k, \widehat{k}\in\{0,1,2,3\}: \quad 
    x= \tau\cdot i^k= (\tau+\lambda)\cdot i^{\widehat{k}}.
  \]
  That is, $\tau=i^k(\tau+\lambda)$ for some $k\in\{0,1,2,3\}$. Therefore, $N_0\cap N_1\neq\emptyset$ implies $N_0= N_1$ and we only have $4$ digits. 
  
\end{proof}

\begin{remark}\label{lt85}
  Let $S(z)=\lambda z+\tau$ be a non-singular linear mapping on the complex plane and $S(\Lambda)$ be the corresponding transformed L\'evy dragon. Write $d$ for the number of digits in the revolving structure of $S(\Lambda)$.
  If $d=4$, then $\tau=i^k(\tau+\lambda)$ for some $k\in\{1,2,3\}$.
\end{remark}

\begin{proof}
  By Lemma \ref{lt90}, $\exists k\in\mathbb{N}:\: i^k(\lambda+\tau)=\tau$.
  It follows from the periodicity of the powers of $i$ that we may cover all cases by restricting ourselves to $k\in\{0,1,2,3\}$. However, since $\tau=\lambda +\tau$ leads to a contradiction, we must have $k\in\{1,2,3\}$.
\end{proof}

We may also give a geometric characterization of linear transformation $S$ based on the number of digits in the revolving structure of $S(\Lambda)$. 

\begin{theorem}\label{lt91}
  Let $S(z)=\lambda z+\tau$ be a non-singular linear mapping on the complex plane and $S(\Lambda)$ be the corresponding transformed L\'evy dragon. Write $d$ for the number of digits in the revolving structure of $S(\Lambda)$. Then, 
  \begin{enumerate}[(a)]
    \item $d=4$ if and only if the axis of symmetry of $S(\Lambda)$ crosses the origin;
    \item $d=5$ if and only if the fixed point of either $\widehat{f}_0$ or $\widehat{f}_1$ is $0$;
    \item otherwise, $d=8$. 
  \end{enumerate}

\end{theorem}

\begin{minipage}{0.45\textwidth}
  \centering
  \includegraphics[scale=.45]{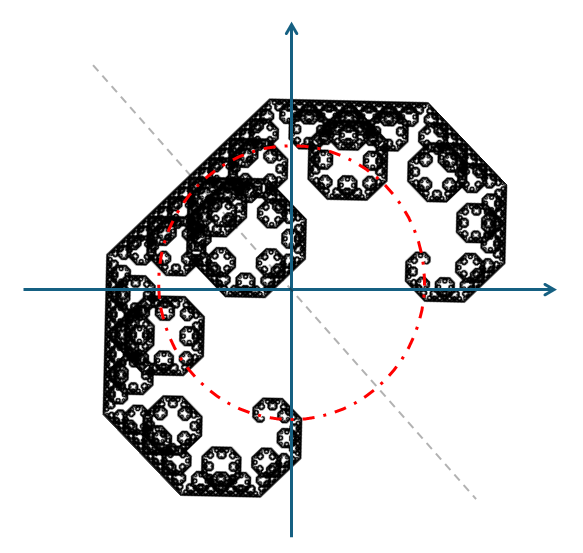}
\end{minipage}
\hspace{0.05\textwidth} 
\begin{minipage}{0.45\textwidth}
  When the transformed L\'evy dragon's axis of symmetry goes through the origin, the fixed points of $\widehat{f}_0$ and $\widehat{f}_1$ are on the same circle around $0$. 
  In this case the revolving structure has only $4$ distinct digits.
\end{minipage}

\begin{minipage}{0.45\textwidth}
  If the fixed point of either $\widehat{f}_0$ or $\widehat{f}_1$ is at $0$, there are only $5$ distinct digits in the revolving structure. The original L\'evy dragon is also in this category.
\end{minipage}
\hspace{0.05\textwidth} 
\begin{minipage}{0.45\textwidth}
  \centering
  \includegraphics[scale=.45]{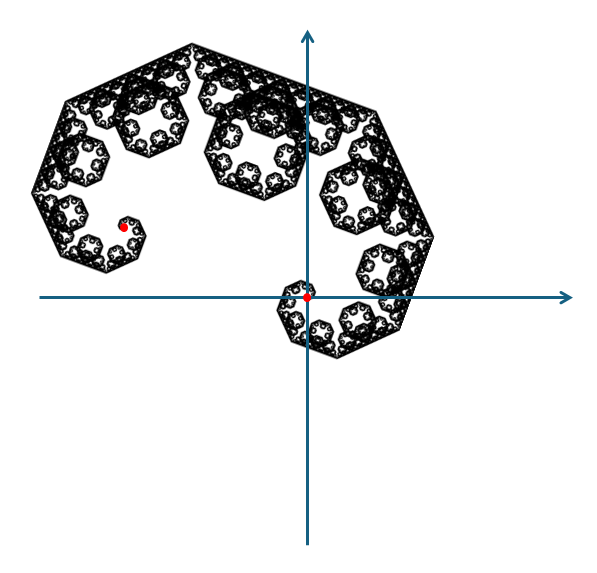}
\end{minipage}

\begin{minipage}{0.45\textwidth}
  \centering
  \includegraphics[scale=.45]{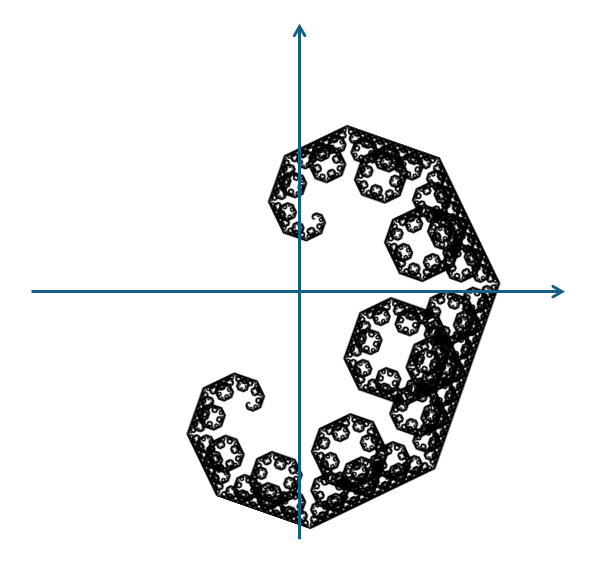}
\end{minipage}
\hspace{0.05\textwidth} 
\begin{minipage}{0.45\textwidth}
  If the transformed L\'evy dragon is in a general position not submitting to the aforementioned alignments, then all $8$ digits in the revolving structure are different values.
\end{minipage}

\begin{proof}

  First, we need to calculate the fixed points of $\widehat{f}_0$ and $\widehat{f}_1$. Recall that 
  \[
    \widehat{f}_0(z)= \frac{1-i}{2}z+\frac{1+i}{2}\tau, \quad
    \widehat{f}_1(z)= \frac{1+i}{2}z+\frac{1-i}{2}(\lambda+\tau). 
  \]
  It is straightforward that $\widehat{f}_0(\tau)=\tau$ and $\widehat{f}_1(\lambda+\tau)=\lambda+\tau$. To prove the theorem, we are left to show that the transformed L\'evy dragon $S(\Lambda)$ is in a certain position relative to the origin when $\tau$ and $\lambda+\tau$ satisfy the corresponding assumption given in Lemma \ref{lt90}
  
  Assume that $d=4$. By Remark \ref{lt85}, there exists $k\in\{1,2,3\}$ such that $i^k(\lambda+\tau)=\tau$. It follows that $|\lambda+\tau|=|\tau|$, hence they are on the circumference of some circle $C$ centered at the origin. 
  The axis of symmetry of $S(\Lambda)$ is the perpendicular bisector of the line segment between $\lambda+\tau$ and $\tau$. This axis must go through the center of $C$, which is the origin. 

  Since contractive functions have unique fixed points, $f_0'$ or $f_1'$ having a fixed point at $0$ implies that either $\tau=0$ or $\lambda+\tau=0$. By Lemma \ref{lt90}, this is equivalent to $d=5$.

  We conclude the proof by referring to part (c) of Lemma \ref{lt90}.

\end{proof}

\section{Future directions}

We plan to extend our analysis to all self-similar attractors on the plane that admit a revolving structure. To do that, first we need to find conditions on the functions in the IFS that guarantee that the attractor has a revolving structure.

Thanks to the symmetry of the L\'evy dragon, we could describe every isometric image of $\Lambda$ with the help of linear transformations. In the general case, we cannot assume any symmetry, and hence we must incorporate reflections separately into our investigation. 

Further, we plan to study the change in the revolving structure under more complex transformations as well.

\end{document}